\documentclass[
final
]{dmtcs-episciences}

\usepackage[utf8]{inputenc}
\usepackage{subfigure,amsmath}
\usepackage{graphicx, tikz} 

\newcommand{\NN}{\mathbb{N}}
\newcommand{\PP}{\mathbb{N}^+}

\newcommand{\mO}{\mathcal{O}}
\newcommand{\PF}{\mathrm{PF}}
\newcommand{\FR}[1]{\mathrm{FR}_{#1}}
\newcommand{\PA}[1]{\mathrm{PA}_{#1}}

\newcommand{\IPF}{\mathrm{\mathrm{IPF}}}

\newcommand{\Fub}{\mathrm{Fub}}

\newcommand*{\Out}{\mathcal{O}}

\newcommand{\upf}[1]{\mathrm{UPF}_{#1}}

\newcommand{\fr}[1]{\mathrm{FR}_{#1}}

\newcommand{\ind}{\mathrm{index}}

\newcommand{\UIPF}[1]{\mathrm{\mathrm{UIPF}}_{#1}}
\newcommand{\NDIPF}[2]{\mathrm{\mathrm{IPF}}_{#1}^{\uparrow}(#2)}

\newcommand{\NDPF}[1]{\mathrm{\mathrm{PF}}_{#1}^{\uparrow}}
\newcommand{\Dyck}[2]{\mathrm{D}_{#1}^{#2}}

\newcommand{\BPA}{\mathrm{BPA}}

\usepackage{thmtools}
\newtheorem{theorem}{Theorem}[section]

\newtheorem{corollary}[theorem]{Corollary}

\newtheorem{remark}[theorem]{Remark}
\newtheorem{example}[theorem]{Example}
\newtheorem{definition}[theorem]{Definition}
\usepackage[round]{natbib}

\author [T. Aguilar-Fraga et al.]{Tom\'{a}s Aguilar-Fraga\affiliationmark{1}
  \and Jennifer Elder\affiliationmark{2}\thanks{J.~Elder was partially supported through an AWM Mentoring Travel Grant.}
  \and Rebecca E.~Garcia\affiliationmark{3}\\
  \and Kimberly P. Hadaway\affiliationmark{4}
  \and Pamela E. Harris\affiliationmark{5}\thanks{P.~E.~Harris was supported through a Karen Uhlenbeck EDGE Fellowship.}
  \and Kimberly J. Harry\affiliationmark{5}\\
  \and Imhotep B. Hogan\affiliationmark{6}
  \and Jakeyl Johnson\affiliationmark{7}
  \and Jan Kretschmann\affiliationmark{5}\\
  \and Kobe Lawson-Chavanu\affiliationmark{8}
  \and J. Carlos Mart\'{i}nez Mori\affiliationmark{9}\thanks{J.~C. Mart\'inez Mori is supported by Schmidt Science Fellows, in partnership with the Rhodes Trust.}
  \and Casandra D. Monroe\affiliationmark{10}\thanks{C.~D. Monroe was partially supported by NSF GRFP Fellowship.}
  \and Daniel Qui\~nonez\affiliationmark{11}
  \and Dirk Tolson III\affiliationmark{12}
  \and Dwight Anderson Williams II\affiliationmark{13}
  }
\title[Interval and $\ell$-interval rational parking functions]{Interval and $\ell$-interval rational parking functions}
\affiliation{

${{}^{1}}$ Department of Mathematics, Harvey Mudd College, Claremont, CA, USA\\
${{}^{2}}$ Department of Computer Science, Mathematics and Physics, Missouri Western State University, MO, USA\\
${{}^{3}}$ Department of Mathematics and Computer Science, Colorado College, Colorado Springs, CO, USA\\
${{}^{4}}$ Department of Mathematics, Iowa State University, Ames, IA, USA\\
${{}^{5}}$ Department of Mathematical Sciences, University of Wisconsin-Milwaukee, Milwaukee, WI, USA\\
${{}^{6}}$ Department of Mathematics, Florida A\&M University, Tallahassee, FL, USA\\
${{}^{7}}$ Department of Mathematics, San Francisco State University, San Francisco, CA, USA\\
${{}^{8}}$ Department of Mathematics, Emory College, Atlanta, GA, USA\\
${{}^{9}}$ H.\thinspace Milton\thinspace Stewart\thinspace School\thinspace of\thinspace Industrial\thinspace and\thinspace Systems\thinspace Engineering, \thinspace Georgia\thinspace Institute\thinspace of\thinspace Technology, Atlanta,\thinspace GA,\thinspace USA\\
${{}^{10}}$ Department of Mathematics, University of Texas at Austin, Austin, TX, USA\\
${{}^{11}}$ Department of Mathematics, UC Irvine, Irvine, CA , USA\\
${{}^{12}}$ Department of Mathematics and statistics, Sonoma State University, Rohnert Park, CA, USA\\
${{}^{13}}$ Department of Mathematics, Morgan State University, Baltimore, MD, USA}

\begin{document}
\publicationdata{vol. 26:1, Permutation Patterns 2023}{2024}{10}{10.46298/dmtcs.12598}{2023-11-27; 2023-11-27; 2024-05-14; 2024-08-30}{2024-08-30}
\maketitle

\begin{abstract}
Interval parking functions are a generalization of parking functions in which cars have an interval preference for their parking. 
We generalize this definition to parking functions with $n$ cars and $m\geq n$ parking spots, which we call interval rational parking functions and provide a formula for their enumeration.
By specifying an integer parameter $\ell\geq 0$, we 
then consider the subset of interval rational parking functions in which each car parks at most $\ell$ spots away from their initial preference. 
We call these $\ell$-interval rational parking functions and provide recursive formulas to enumerate this set for all positive integers $m\geq n$ and $\ell$.
We also establish formulas for the number of nondecreasing $\ell$-interval rational parking functions via the outcome map on rational parking functions.
We also consider the intersection between $\ell$-interval parking functions and Fubini rankings and show the  enumeration of these sets is given by generalized Fibonacci numbers. 
We conclude by specializing $\ell=1$, and establish that the set of $1$-interval rational parking functions with $n$ cars and $m$ spots are in bijection with the set of barred preferential arrangements of $[n]$ with $m-n$ bars. This readily implies enumerative formulas.
Further, in the case where $\ell=1$, we recover the results of Hadaway and Harris that unit interval parking functions are in bijection with the set of Fubini rankings, which are enumerated by the Fubini numbers.
\end{abstract}

%

\section{Introduction}\label{sec:intro}
Throughout, we let $\NN =\{0, 1,2,3,\ldots\}$, $\PP=\{1,2,3,\ldots\}$, and for $n\in\PP$, we let $[n]=\{1,2,\ldots, n\}$. 
We begin by recalling a definition of parking functions, as introduced by \cite{konheimOccupancyDisciplineApplications1966}. 
With $n\in\PP$, a tuple $\alpha=(a_1,a_2,\ldots,a_n)\in[n]^n$ is a \textit{parking function of length $n$} if its nondecreasing rearrangement $\alpha'=(a_1',a_2',\ldots,a_n')$ satisfies $a_i'\leq i$ for all $i\in[n]$. 
For example, $\alpha=(3,1,1,3,4,5)$ is a parking function of length six as its nondecreasing rearrangement $\alpha'=(1,1,3,3,4,5)$ satisfies the required inequality conditions.
However, $\alpha=(2,3,3,3)$ is not a parking function, as its nondecreasing rearrangement 
$\alpha'=(2,3,3,3)$ does not satisfy the inequality conditions: In particular, $a_1=2$ which is not less than or equal to $1$. 
Let $\PF_n \subseteq [n]^n$ denote the set of parking functions of length $n$. Several authors have provided proofs that 
$|\PF_n|=(n+1)^{n-1}$, see~\cite{konheimOccupancyDisciplineApplications1966,pyke,Pollock}.

An alternative definition of parking functions is as follows. 
Let the tuple $\alpha=(a_1,a_2,\ldots,a_n) \in [n]^n$ 
encode the preferred parking spots of $n$ cars, $a_i\in[n]$  being the parking preference of car $i$, as the cars attempt to park one after the other on a one-way street. In this sense,  $\alpha$ is called a \textit{preference list}.
Throughout the paper, when we refer to the \emph{standard parking procedure}, we refer to the process that a car employs whenever it enters the street in its attempt to park. 
Namely, when car $i \in [n]$ arrives, it attempts to park in its preferred spot $a_i$.
If spot $a_i$ is unoccupied, car $i$ parks there.
Otherwise, car $i$ continues driving down the one-way street parking in the first unoccupied spot it encounters.
If no such spot exists, then car $i$ is unable to park.
If all cars are able to park given the preference $\alpha$, then $\alpha$ is a parking function of length $n$. 
Figure~\ref{fig: parking} illustrates the order in which cars park on the street when $\alpha=(3,1,1,3,4,5)\in\PF_{6}$. 
When $\alpha\in\PF_n$, we refer to the resulting parking order of the cars under $\alpha$ as the \emph{outcome} of $\alpha$, which we denote by $\mO(\alpha)$. 
For $\alpha\in\PF_n$, the outcome $\mO(\alpha)$ is an element of $\mathfrak{S}_n$, the set of permutations of $[n]$.

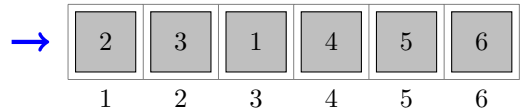
\begin{figure}[h]
    \centering
    \begin{tikzpicture}
    \draw[step=1cm,gray,very thin] (0,0) grid (6,1);
    \draw[ultra thick,->,blue] (-.75,.5) -- (-.25,.5);
    \draw[fill=gray!50] (.1,0.1) rectangle (.9,.9);
    \node at (.5,.5) {$2$};
    \draw[fill=gray!50] (1.1,0.1) rectangle (1.9,.9);
    \node at (1.5,.5) {$3$};        
    \draw[fill=gray!50] (2.1,0.1) rectangle (2.9,.9);
    \node at (2.5,.5) {$1$};
    \draw[fill=gray!50] (3.1,0.1) rectangle (3.9,.9);
    \node at (3.5,.5) {$4$};
    \draw[fill=gray!50] (4.1,0.1) rectangle (4.9,.9);
    \node at (4.5,.5) {$5$};
    \draw[fill=gray!50] (5.1,0.1) rectangle (5.9,.9);
    \node at (5.5,.5) {$6$};
    
    \node at (.5,-.25) {$1$};
    \node at (1.5,-.25) {$2$};
    \node at (2.5,-.25) {$3$};
    \node at (3.5,-.25) {$4$};
    \node at (4.5,-.25) {$5$};
    \node at (5.5,-.25) {$6$};

    \end{tikzpicture}
    \caption{
    The parking outcome of the parking function $\alpha=(3,1,1,3,4,5)\in \PF_6$.
    }
    \label{fig: parking}
\end{figure}

Parking functions have numerous applications and connections to many fields in mathematics, including computing volumes of flow polytopes, hyperplane arrangements, and the sorting algorithms \cite{beck2014parking, Benedetti_2019, harris2023lucky}. 
Work on parking functions includes the study of discrete statistics \cite{adeniran2023pattern,GesselSeo,Schumacher_Descents_PF}, families of parking functions with particular desirable properties \cite{Hanoi,flat_pf}, and many generalizations of parking functions \cite{colaric2020interval,knaples,assortments,MVP}. 
Our work is motivated by two generalizations of parking functions, which we refer to as ``rational'' parking functions and interval parking functions. We describe those next.

First we consider the situation in which there are $n\in\PP$ cars and  $m\geq n$ parking spots on a one-way street. In this case, a preference $\alpha$ is an element of $[m]^n$, and the set  $\PF_{n,m}$ consists of the preferences allowing all $n$ cars to park under the standard parking procedure. This generalization of parking functions is known as $(m,n)$-parking functions. In the case of $n$ and $m$ being coprime, $(m,n)$-parking functions are referred to as rational parking functions, see \cite{armstrong2016rational}. In the current work, we simply refer to $(m,n)$-parking functions as rational parking functions without requiring $m$ and $n$ to be coprime. These parking functions were enumerated during the onset of parking function studies \cite[Lemma 2]{konheimOccupancyDisciplineApplications1966} 
and, implicitly, in the statistical work of \cite{pyke}, giving
\begin{equation}
|\PF_{n,m}| =  
\left(m+1-n\right)(m+1)^{n-1}.
\label{rational count}
\end{equation}

Another generalization of parking functions are \emph{interval} parking functions which describe the case where cars have an interval of the street in which they want to park. 
Formally, interval parking functions are defined by a pair $(\alpha,\beta)$, where $\alpha=(a_1,a_2,\ldots,a_n)\in\PF_n$ and $\beta=(b_1,b_2,\ldots,b_n)\in[n]^n$, satisfy, for all $i\in[n]$, the inequality $a_i\leq b_i\leq n$.
For $i\in[n]$, we call $a_i$ the \textit{initial preference} of car $i$, $b_i$ the \textit{tolerance} of car $i$ (as it is the highest numbered parking spot in which car $i$ would tolerate parking), and the set $\{a_i,a_i+1,\ldots,b_i\}$ the \textit{preference interval} for car $i$. Throughout, we refer to $\beta$ as the \emph{tolerance vector}.
If, for each $i\in[n]$, car $i$ is able to park in its preference interval following the standard parking procedure, then the pair $(\alpha,\beta)$ is called an \textit{interval parking function}. 
We let $\IPF_n$ denote the set of interval parking functions of length $n$.
Colaric, DeMuse, Martin, and Yin established~\cite[Proposition 3.2]{colaric2020interval}:
\begin{equation}
|\IPF_n|=n!(n+1)^{n-1}.\label{interval count}
\end{equation}
Note that for $\beta=(n,n,\ldots,n)\in[n]^n$, the set of interval parking functions is precisely the set of parking functions.

In the present work,
we begin by generalizing interval parking functions to their rational analogue, which we call \textit{interval rational parking functions}. Here, there are $n$ cars and $m\geq n$ parking spots on the street. 
In this case, we consider the pair of tuples $(\alpha,\beta)$, where  
$\alpha=(a_1,a_2,\ldots,a_n)\in\PF_{n,m}$ and $\beta=(b_1,b_2,\ldots,b_n)\in[m]^n$ satisfying $a_i\leq b_i\leq m$. 
If, for each $i\in[n]$, car $i$ is able to park in one of the spots numbered $a_i,a_i+1,\ldots,b_i$ (following the standard parking procedure), then $(\alpha,\beta)$ is an interval rational parking function. 
We let $\IPF_{n,m}$ denote the set of interval rational parking functions with $n$ cars and $m$ spots. 
For example, $((3,5,7,1,1),(3,6,8,1,4))\in \IPF_{5,8}$ is a rational interval parking function with five cars and eight parking spots whose parking outcome is illustrated in Figure \ref{fig:interval rational example}. 
However, $((3,5,7,1,1),(3,6,8,1,1))$ is not an interval rational parking function for any $\IPF_{5,m}$ with $m\geq 8$, as car $5$ would fail to park in the first spot, which is the only spot in its preference interval. In our first result, we establish a formula for $|\IPF_{n,m}|$ as a sum over rational parking functions involving the parking outcome, i.e.~the order in which the cars park on the street (Theorem~\ref{thm:irpf count}).  Specializing $m=n$ recovers the formula in Equation (\ref{interval count}) for the number of interval parking functions.

\begin{figure}[h]
    \centering
    \begin{tikzpicture}
    \draw[step=1cm,gray,very thin] (0,0) grid (8,1);
    \draw[ultra thick,->,blue] (-.75,.5) -- (-.25,.5);
    \draw[fill=gray!50] (.1,0.1) rectangle (.9,.9);
    \node at (.5,.5) {$4$};
    \draw[fill=gray!50] (1.1,0.1) rectangle (1.9,.9);
    \node at (1.5,.5) {$5$};   
    \draw[fill=gray!50] (2.1,0.1) rectangle (2.9,.9);
    \node at (2.5,.5) {$1$};
     \draw[fill=gray!50] (4.1,0.1) rectangle (4.9,.9);
     \node at (4.5,.5) {$2$};
     \draw[fill=gray!50] (6.1,0.1) rectangle (6.9,.9);
      \node at (6.5,.5) {$3$};
    \node at (.5,-.25) {$1$};
    \node at (1.5,-.25) {$2$};
    \node at (2.5,-.25) {$3$};
    \node at (3.5,-.25) {$4$};
    \node at (4.5,-.25) {$5$};
    \node at (5.5,-.25) {$6$};
    \node at (6.5,-.25) {$7$};
    \node at (7.5,-.25) {$8$};

    \end{tikzpicture}
    \caption{The parking outcome of the interval rational parking function $((3,5,7,1,1),(3,6,8,1,4))\in \IPF_{5,8}$.}
    \label{fig:interval rational example}
\end{figure}
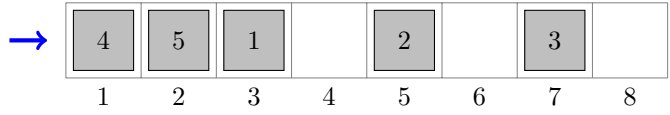

Next, we consider the subset $\IPF_{n,m}(\ell)\subseteq \IPF_{n,m}$, which is defined via an integer parameter $\ell\geq 0$, and we refer to it as the set of $\ell$-interval rational parking functions. 
In this setting, every car must park in its first preferred spot or at most $\ell$ spots away from it. 
That is, if $\alpha=(a_1,a_2,\ldots,a_n)\in\PF_{n,m}$, and  
if, for each $i\in[n]$, car $i$ is able to park in one of the spots $a_i,a_{i+1},\ldots,  \min(a_i+\ell,m)$ (following the standard parking procedure), then $\alpha$ is an $\ell$-interval rational parking function. 
If $m=n$, then we let $\IPF_{n}(\ell)$ be the set of $\ell$-interval parking functions of length $n$.
The following is a list of our main results as related to $\ell$-interval rational parking functions.
\begin{enumerate}
    \item We establish recursive formulas for the number of $\ell$-interval rational parking functions with $n$ cars and $m$ parking spots (Theorem~\ref{them:rational recursion} for $n=m$ and Theorem~\ref{thm:LIntRec} for $n>m$). 
        \item We give a recursive formula for the number of nondecreasing $\ell$-interval rational parking functions with $n$ cars and $m $ parking spots (Theorem \ref{thm:nondecreasing ell rational}).
        \item 
        We define a bijection between the set of nondecreasing $\ell$-interval rational parking functions with $n$ cars and $m$ parking spots and the set of $(n,m)$-Dyck paths with height $\ell+1$ (Theorem \ref{them:bijection nondecreasing to dyck paths} when $n=m$, and Theorem \ref{rational pf to dyck paths} when $m\geq n$). If $n=m$, this yields a enumeration through generating functions (Corollary \ref{cor:nondec_form}) and an open problem for $m\geq n$.

    \end{enumerate}

Fubini rankings and unit interval parking functions have received much recent attention in the literature. 
In \cite[Theorem 2.9]{unit_pf}, the authors give a complete characterization of unit interval parking functions. 
In \cite{unit_perm}, the authors establish that unit interval parking functions with $k$ unlucky cars (cars that do not park in their preference) are in bijection with the $k$-dimensional faces of the permutohedron of order $n$. 
In \cite{boolean}, the authors consider the set of unit Fubini rankings, which is the intersection of the set of unit interval parking functions of length $n$ and the set of Fubini rankings.
Through the set of unit Fubini rankings and their parking interpretation, they establish a complete characterization and enumeration for the total number of Boolean intervals in the weak Bruhat order of the symmetric group on $n$ letters, and the total number of Boolean intervals of rank $k$ in the same.
Motivated by these results, we establish the following.    
\begin{enumerate}
\item[(4)] When $m=n$, we show that the set of nondecreasing $\ell$-interval parking functions which are also Fubini rankings (Definition~\ref{def:fubini_rank}) are enumerated by a Fibonacci sequence of order $\ell+1$ (Theorem~\ref{thm:FibonacciOrdern}).
    
    \item[(5)] When $\ell=1$, we provide:
    \begin{enumerate}
    \item A bijection between unit interval rational parking functions with $n$ cars and $m$ parking spots  and barred preferential arrangements with $m-n$ bars (Theorem \ref{thm:bijection_uirpf}).
    \item The specialization of $m=n$ is provided in Theorem \ref{lem:funini and pa}, recovering the fact that unit interval parking functions are enumerated by the Fubini numbers \cite{Hadaway_unit_interval}.
\end{enumerate}
\end{enumerate}

\section{Background}\label{sec:background}

In this section, we provide background related to rational parking functions, interval parking functions, and unit interval parking functions. 

\subsection{Rational parking functions}
In this scenario, there are $n$ cars and $m\geq n$ parking spots.
A preference list $\alpha\in[m]^n$ permitting $n$ cars to park in $m$ spots utilizing the standard parking procedure is called a \emph{rational parking function} of order $(n,m)$. 
Throughout, we let $\PF_{n,m}$ denote the set of rational parking functions of order $(n,m)$. As the order of the parking function is presented as a subscript, we henceforth omit stating the order of a rational parking function as it will be clear from context. 
Recall, that Equation (\ref{rational count}) gives $\left|\PF_{n,m}\right| =  
\left(m+1-n\right)(m+1)^{n-1}$. 
Moreover, if $n=m$, then $\PF_{n,n} = \PF_{n}$, and thus $\PF_n \subseteq \PF_{n,m}$, for all $m\geq n$.
In this way, rational parking functions generalize parking functions.

\subsection{Interval parking functions}
We begin by remarking that in \cite{colaric2020interval}, interval parking functions were only defined for $m=n$. 
In this setting, interval parking functions are characterized as a pair $(\alpha,\beta)\in \PF_n\times [n]^n$, where $\alpha=(a_1,a_2,\ldots,a_n)\in\PF_n$ and $\beta=(b_1,b_2,\ldots,b_n)\in[n]^n$ satisfies $b_i \geq a_i$ for all $1\leq i\leq n$ \cite[Proposition 3.2]{colaric2020interval}. 
Given $\alpha\in \PF_n$ and $\beta\in[n]^n$, 
car $i$ enters the street and using the standard parking procedure, attempts to park between spots $a_i$ and $b_i$ (inclusive), which we refer to as its \emph{preferred interval}.
If all cars can park within their preferred interval utilizing the standard parking procedure, then the pair  $(\alpha,\beta)$ is an interval parking function. 
We let $\IPF_n$ denote the set of interval parking functions of length $n$ and recall Equation (\ref{interval count}) gives $|\IPF_{n}|=n!(n+1)^{n-1}$. 

Note that if $\beta=(b_1,b_2,\ldots,b_n)\in[n]^n$ and 
$(\alpha,\beta)\in\IPF_n$, then 
$(\alpha,\beta')\in\IPF_n$, for all 
$\beta'=(b_1', b_2',\ldots , b_n') \in [n]^n$ satisfying $b_i \leq b'_i$, for all $i\in[n]$. 
In this case, $\beta'$ allows the cars to possibly park further down the street. 
However, the converse is not necessarily true: for any interval parking function $(\alpha, \beta)\in\IPF_n$, let 
$\widehat{\beta}=(\widehat{b}_1,\widehat{b}_2,\ldots,\widehat{b}_n)$ be such that $b_i \geq \widehat{b}_i$ for all $i\in[n]$.
Then $(\alpha,\widehat{\beta})$ may not be an interval parking function, as the tolerance vector has possibly reduced the allowable parking spots a car may tolerate. 
We illustrate this with the following example. 
\begin{example}
Consider $(\alpha, \beta) = \left( \left(1, 2, 2 \right), \left(2, 3, 3 \right) \right)$, which one can readily verify is in $\IPF_3$. 
If $\widehat{\beta}=  \left(2, 3, 2 \right)$, then $(\alpha,\widehat{\beta})$ is not an interval parking function since the third car fails to park in the second spot, which is its only preferred parking spot.
\end{example}

\subsection{Unit interval parking functions}
Hadaway and Harris introduced unit interval parking functions, which are a subset of $\PF_n$ in which all cars park in their preferred spot or one spot further down the street~\cite{Hadaway_unit_interval}. 
Let $\upf{n}$ denote the set of unit interval parking functions of length~$n$.
Note, $(1, 2, 3, 4, 5)$, $(1, 1, 3, 4, 5)$, $(1, 1, 2, 4, 5)$ are elements of $\upf{5}$, whereas $(1, 1, 1, 1, 1)$ is not a unit interval parking function as car three cannot park 
in its preferred spot nor the one after it.

Hadaway and Harris established the following enumerative result. 
\begin{theorem}{\cite[Theorem~5.12]{Hadaway_unit_interval}}\label{thm:Hadaway}
If $n\geq 1$, then
\begin{equation}\label{eq:fubini numbers}
 |\upf{n}|
=
 \Fub_{n}
\end{equation}
where $\Fub_n$ denotes the $n$th Fubini number (OEIS \href{https://oeis.org/A000670}{A000670}), which are known to satisfy
\begin{equation}\label{eq:fubini numbers2}
 \Fub_{n}
  =\sum_{k=1}^n k!\, S(n,k),
\end{equation}
where $S(n,k)$ are Stirling numbers of the second kind (OEIS \href{https://oeis.org/A008277}{A008277}) and count the number of set partitions of $[n]$ with $k$ nonempty parts.
\end{theorem}

To establish Theorem \ref{thm:Hadaway}, Hadaway and Harris proved that the set of unit interval parking functions is in bijection with the set of \emph{Fubini rankings} \cite[Theorem 5.12]{Hadaway_unit_interval}. 
As Fubini rankings play a role in our work, we define them next.

\begin{definition}\label{def:fubini_rank}
    A \emph{Fubini ranking of length $n$} is a tuple $r=(r_1,r_2,\ldots,r_n)\in [n]^n$ that records a valid ranking of $n$ competitors with ties allowed (i.e., multiple competitors can be tied and have the same rank). If $k$ competitors are tied for rank $i$, the $k-1$ subsequent ranks $i + 1, i + 2, \ldots, i + k - 1$ are disallowed.\footnote{As noted in \cite{boolean}: ``One noteworthy instance of the tie condition of a Fubini ranking took place at the men's high jump event at the Summer 2020 Olympics~\cite{olympics}.
In this competition, Mutaz Essa Barshim of Qatar and Gianmarco Tamberi of Italy led the final round.
Both athletes cleared 2.37 meters, but neither of them cleared 2.39 meters.
Upon being presented the option of a ``jump-off'' to determine the sole winner, they agreed to instead share the gold medal. 
The next best rank was held by Maksim Nedasekau of Belarus, who obtained the bronze medal.''}
\end{definition}

For example, $(8,1,2,5,6,2,2,6)$ is  Fubini ranking, in which competitor one ranks eighth, competitor two ranks first, competitors three, six, and seven tie at rank two, competitor four ranks fifth, and competitors five and eight rank sixth. 
Note $(1,1,1,2)$ is not a Fubini ranking as the triple tie for rank one would disallow ranks 2 and 3, making rank 4 the next available. 
We let $\fr{n}$ denote the set of Fubini rankings of length $n$. We remark that 
\cite{cayley_2009} showed that $|\fr{n}|=\Fub_n$; this together with the bijection of Hadway and Harris, establishes Equation (\ref{eq:fubini numbers}).

\section{Interval and \texorpdfstring{$\ell$-}-interval rational parking functions}\label{sec:LIntervalRec}
In this section, we give a formula for the number of interval rational parking functions. 
We then give recursive formulas for the number of $\ell$-interval rational parking functions by considering the cases where $m=n$ and $m>n$ independently.

\subsection{Interval rational parking functions.}
Consider 
the set of interval rational parking functions, denoted by $\IPF_{n,m}$, in which there are $n$ cars and $m\geq n$ parking spots. 
Elements of $\IPF_{n,m}$ are pairs of tuples $(\alpha,\beta)$, that allow 
all cars to park within their tolerance 
when utilizing the standard parking procedure.
We again call $\beta$ the tolerance vector, as defined for interval parking functions. 

We now formally define the \textit{outcome} of a rational parking function. 

\begin{definition}\label{def:outcome}
For a fixed number of spots $m \in \PP$, the \emph{outcome map} of a rational parking function is the function
$\Out_m: \PF_{n,m}\to \{0,1,\ldots,n\}^m$
 given by 
$\Out_m(\alpha)=(\sigma_1,\sigma_2,\ldots,\sigma_m)$
where \[\sigma_i=
\begin{cases}
    j\mbox{if car $j$ parks in spot $i$ under $\alpha$}\\
    0\mbox{if no car parks in spot $i$ under $\alpha$}.
\end{cases}\]
Note that in the outcome map we omit the duplicate parenthesis to simplify the notation, i.e. we write $\Out_m(\alpha)=\Out_m(\alpha_1,\alpha_2\ldots,\alpha_n)$ instead of $\Out_m((\alpha_1,\alpha_2,\ldots,\alpha_n))$.
\end{definition}

We provide the following example to illustrate the outcome map.

\begin{example}\label{ex:rational_outcome}
The parking outcome of 
$(4,3,5,5,9)\in\PF_{5,10}$ is illustrated in Figure~\ref{fig:interval rational outcome example}. 
Based on Figure~\ref{fig:interval rational outcome example} and by Definition \ref{def:outcome}, we have $\Out_{10}(4,3,5,5,9)=(0,0,2,1,3,4,0,0,5,0)$.
\begin{figure}[h]
    \centering
    \begin{tikzpicture}
    \draw[step=1cm,gray,very thin] (0,0) grid (10,1);
    \draw[fill=gray!50] (2.1,0.1) rectangle (2.9,.9);
    \node at (2.5,.5) {$2$};
    \draw[fill=gray!50] (3.1,0.1) rectangle (3.9,.9);
    \node at (3.5,.5) {$1$};   
    \draw[fill=gray!50] (4.1,0.1) rectangle (4.9,.9);
    \node at (4.5,.5) {$3$};
     \draw[fill=gray!50] (5.1,0.1) rectangle (5.9,.9);
     \node at (5.5,.5) {$4$};
     \draw[fill=gray!50] (8.1,0.1) rectangle (8.9,.9);
      \node at (8.5,.5) {$5$};
    \node at (.5,-.25) {$1$};
    \node at (1.5,-.25) {$2$};
    \node at (2.5,-.25) {$3$};
    \node at (3.5,-.25) {$4$};
    \node at (4.5,-.25) {$5$};
    \node at (5.5,-.25) {$6$};
    \node at (6.5,-.25) {$7$};
    \node at (7.5,-.25) {$8$};
    \node at (8.5,-.25) {$9$};
    \node at (9.5,-.25) {$10$};

    \end{tikzpicture}
    \caption{The parking outcome of the interval rational parking function $(4,3,5,5,9)\in\PF_{5,10}$.}
    \label{fig:interval rational outcome example}
\end{figure}
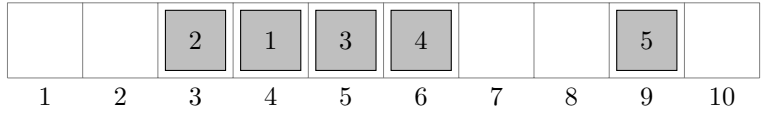

\end{example}

It will be important to know the parking spots that cars occupy after they all park with a given preference list $\alpha\in\PF_{n,m}$.
To keep track of this information, we let \begin{equation}
    I(\alpha)=\{i\in[m]: \sigma_i\neq 0 \mbox{ where } \Out_m(\alpha)=(\sigma_1,\sigma_2,\ldots,\sigma_m)\}.\label{nonempty indices}
\end{equation}
Note that for all $\alpha\in\PF_{n,m}$, we have $|I(\alpha)|=|[n]|=n$, as $I(\alpha)$ keeps track of the spots occupied by the $n$ cars when parking under $\alpha$. 
\begin{example}[Continuing Example \ref{ex:rational_outcome}]
    If $\alpha=(4,3,5,5,9)\in\PF_{5,10}$, then $I(\alpha)=\{3,4,5,6,9\}$.
\end{example}

We now enumerate all interval rational parking functions. 

\begin{theorem}\label{thm:irpf count}
If $m\geq n\geq 1$, then
\[\left|\IPF_{n,m}\right| = \sum_{\alpha\,\in\,\PF_{n,m}} \left(\prod_{i\,\in\, I(\alpha)} \left( m-i+1\right) \right).\]
\end{theorem}

\begin{proof}
For any $\alpha\in\PF_{n,m}$, let 
$I(\alpha)$ be as in Equation (\ref{nonempty indices}). As we noted above, if $\alpha\in\PF_{n,m}$, then $|I(\alpha)|=n$.
Note that if $i\in I(\alpha)$, 
then 
$\sigma_i=j$ for a unique $j \in [n]$. 
Then $b_j$, the $j$th entry of the tolerance vector $\beta$,  must satisfy $i\leq b_j\leq m$. 
That gives $m-i+1$ options for the possible value of $b_j$.
As the tolerances of cars are independent, we can take the product of the possible values the tolerance can be. 
Thus, there are 
\[ \prod_{i\,\in\, I(\alpha)} \left( m-i+1\right)\]
total tolerances for each $\alpha\in\PF_{n,m}$.
The result follows from taking the sum of this product over all possible rational parking functions  $\alpha\in \PF_{n,m}$.
\end{proof}

By setting $m=n$ in Theorem \ref{thm:irpf count}, we recover Equation (\ref{interval count}). We state and prove this below.
\begin{corollary}{\cite[Proposition 3.2]{colaric2020interval}}\label{cor: count ipf total}
    If $n\geq 1$, then 
    $|\IPF_{n}|=n!(n+1)^{n-1}$.
\end{corollary}
\begin{proof}
    Note that $I(\alpha)=[n]$ for any $\alpha\in\PF_n$. Thus, by Theorem \ref{thm:irpf count}, setting $m=n$ we have that
    \[|\IPF_{n}|=\sum_{\alpha\,\in\,\PF_n}\left(\prod_{i\,\in\, I(\alpha)}(n-i+1)\right)=\sum_{\alpha\,\in\,\PF_n}\left(\prod_{i\,\in\, [n]}(n-i+1)\right)=\sum_{\alpha\,\in\,\PF_n}n!=n!\,(n+1)^{n-1}.\]
\end{proof}

\subsection{\texorpdfstring{$\ell-$}-interval parking functions}
Let $n\in \PP$ and let $\ell\in\NN$. We now consider the $\ell$-interval rational parking functions $\alpha \in \IPF_{n}(\ell)$, 
which are parking functions where every car parks at most $\ell$ spots away from its preferred spot. 
We remark that if $\ell=1$, then $\UIPF{n} = \IPF_n(1)$.
Our main result gives a recursive formula for the number of $\ell$-interval parking functions for any $\ell\in\NN$. 
We begin with the following remark. 
\begin{remark}\label{rem:1}
    If $n \in \PP$ and $0\leq \ell\leq n-1$, then $\IPF_n(\ell)\subseteq \IPF_n(\ell+1)$. Moreover, if $\ell\geq n-1$, then $\IPF_n(\ell)=\PF_n$. 
    To see this, fix $\ell$ satisfying $0\leq \ell\leq n-1$. 
If $\alpha\in\IPF_n(\ell)$, then under $\alpha$ every car parks in a spot that is at most $\ell$ spots away from its preference. 
Thus, under $\alpha$ every car is able to park in a spot that is at most $\ell+1$ spots away from their preference, as this would simply give the cars more tolerance for their parking. 
Thus, $\alpha\in\IPF_n(\ell+1)$ as desired.
Additionally, note that if $\ell \geq n-1$, then every car tolerates parking from their initial preference until the end of the street, as there are $n$ parking spots on the street. 
Thus, in this case $\IPF_n(\ell)=\PF_n$, which recovers the standard parking functions $\PF_n$, whenever $\ell\geq n-1$.
\end{remark}

We now give a recursive formula for the number of $\ell$-interval parking functions of length $n$.

\begin{theorem}\label{thm:LIntRec}
If $n \in \PP$ and $\ell\in\NN$, then 
\[
| \IPF_{n}(\ell) |=\sum^{n-1}_{x=0} \binom{n-1}{x} \cdot \min \left( x+1, \ell+1\right) \cdot |  \IPF_{x}(\ell) | \cdot  |  \IPF_{n-1-x}(\ell) |,
\]
with $\left| \IPF_{0}(\ell) \right| = 1$.
\end{theorem}
\begin{proof}
Fix an arbitrary $z\in[n]$. 
We partition the set $\IPF_{n}(\ell)$ into subsets $S_z$ consisting of the elements in $\IPF_n(\ell)$ in which car $n$ parks in spot $z$.
This partition implies that \[ \left| \IPF_{n}(\ell) \right| = \sum_{z=1}^{n} \left| S_{z} \right|. \]
Note that the first $n-1$ cars (numbered $1$ through $n-1$) have parked according to the standard parking procedure and spot $z$ remains unoccupied.
Thus, there are $z-1$ cars parked to the left of the empty parking spot $z$ and there are $n-z$ cars parked to the right of the empty spot. 
There are $\binom{n-1}{z-1}$ ways to select a subset $J \subset [n-1]$ of cars parking on the $z-1$ spots to the left of the empty spot.
For each such subset $J$, the remaining cars in the set $[n-1]\setminus J$ park to the right of the empty spot.
The number of parking preferences of the cars in $J$ which would result in those cars parking to the left of the empty spot is given by $|\IPF_{z-1}(\ell)|$. 
Similarly, the number of parking preferences of the cars in $[n-1]\setminus J$ which would result in those cars parking to the right of the empty spot is given by $|\IPF_{n-z}(\ell)|$.
Moreover, those sets of preferences are independent of each other. 

Now consider the possible preferences of car $n$ so that it parks in spot $z$. 
Car $n$ could prefer spot $z$ or any of the previous $\ell$ spots (if they exist) which are to the left of spot $z$.
The only case in which the previous $\ell$ spots do not all exist is if $z \leq \ell +1$.
Thus, we must take the minimum of $z$ and $\ell +1$ to find the total number of preferences for car $n$. 
Note that this choice is also independent of all other choices.

Thus, we find that
\[
\left| S_{z} \right| = \binom{n-1}{z-1} \cdot \min \left( z, \ell+1\right) \cdot \left|  \IPF_{z-1}(\ell) \right| \cdot \left|  \IPF_{n-z}(\ell) \right|.
\]
Summing over all $1\leq z\leq n$ and reindexing over $0\leq x\leq n-1$, we obtain
\[
\left| \IPF_{n}{(\ell)} \right|=\sum^{n-1}_{x=0} \binom{n-1}{x} \cdot \min \left( x+1, \ell+1\right) \cdot \left|  \IPF_{x}(\ell) \right| \cdot  \left|  \IPF_{n-x-1}(\ell) \right|,
\]
as desired.
\end{proof}

In Table \ref{tab:seqs}, we give data for the cardinalities of the set $\IPF_n(\ell)$ when  $0\leq n\leq 9$ and $0\leq \ell\leq 7$. Note that when $\ell=0$, $|\IPF_n(0)|$ corresponds to $n!$, as is expected since $\IPF_n(0)=\mathfrak{S}_n$.
When $\ell = 1$, $|\IPF_n(1)|$ corresponds to the Fubini numbers, also known as the ordered Bell numbers (OEIS \href{https://oeis.org/A000670}{A000670}), as expected since $\IPF_n(1)=\upf{n}$ is the set of unit interval parking functions. 
The remaining OEIS sequences linked in Table \ref{tab:seqs} were added to the OEIS by the authors. Moreover, by Remark~\ref{rem:1}, as $\ell$ gets larger, the number in the table stabilize.
\begin{table}[ht]
\centering
   \begin{tabular}{|c||c|c|c|c|c|c|c|c|c|c|} \hline
   
    $\ell\,\setminus\, n$  &1&2&3&4&5&6&7&8&9  \\  \hline\hline
    0 & {1} & 2 & 6 & 24 & 120 & 720 & 5040 & 40320 & 362880 \\ \hline
    1  & {1} & {3} & 13 & 75 & 541 & 4683 & 47293 & 545835 & 7087261 \\ \hline
    2 & {1} & {3} & {16} & 109 & 918 & 9277 & 109438 & 1475691 & 22386070 \\ \hline 
    3  & {1} & {3} & {16} & {125} & 1171 & 12965 & 166836 & 2455121 & 40675881 \\ \hline
    4  & {1} & {3} & {16} & {125} & {1296}& 15511& 212978& 3321091& 58196400 \\ \hline
    5 & {1}& {3}& {16}& {125}& {1296}& {16807}& 245337& 4023383& 73652251 \\ \hline
    6 & {1}& {3}& {16}& {125}& {1296}& {16807}& {262144}& 4520825& 86239758 \\ \hline
    7 & {1}& {3}& {16}& {125}& {1296}& {16807}& {262144}& {4782969}& 95217031\\ \hline
\end{tabular} 
\caption{Cardinalities of the set $ \IPF_{n}(\ell) $ for $0\leq n\leq 9$ and $0\leq \ell\leq 7$, added by the authors to OIES as sequence \href{https://oeis.org/A365623}{A365623}.}
\label{tab:seqs}
\end{table}

\subsection{\texorpdfstring{$\ell-$}-interval rational parking functions} 
Let $m,n\in \PP$ such that $m\geq n$ and let $\ell\in\NN$. We now consider the $\ell$-interval rational parking functions $\alpha \in \IPF_{n,m}(\ell)$, which are rational parking functions where every car parks at most $\ell$ spots away from its preferred spot. 
We generalize Theorem \ref{thm:LIntRec} as follows.

\begin{theorem}\label{them:rational recursion}
If $m,n\in\PP$ with $m\geq n$ and $\ell\in\NN$, then
\[
    |\IPF_{n,m}(\ell)| =\sum_{z=1}^{m}\sum_{k=0}^{z-1}\binom{n-1}{k}\cdot|\IPF_{n-1-k,m-z}(\ell)|\cdot F(k),
\]
where
\[F(k)=\sum_{L=0}^{k}\binom{k}{L}\cdot (\min(\ell,L)+1)\cdot |\IPF_{k-L,z-L-2}(\ell)|\cdot |\IPF_{L}(\ell)|.\]
\end{theorem}

\begin{proof}
We proceed as in the proof of Theorem \ref{thm:LIntRec} and assume that the last car parks in spot $z$. 
There are $m\geq n$ parking spots where $n$ is the number of cars. If car $n$ parks in spot $z$, there are $0\leq k\leq z-1$ cars parked to the left of spot $z$ and the remaining $n-1-k$ cars park to the right of spot $z$.
There are $\binom{n-1}{k}$ ways to select the cars that park to the left of spot $z$.
Before proceeding with the count, we need to consider that the preference of car $n$ which parks in spot $z$ is dependent on the number of {contiguous} cars which occupy spots immediately to the left of spot $z$, as then car $n$ could prefer some of those parking spots. 
Let $L$ be the number of cars parked immediately to the left of spot $z$ (from the $k$ cars parked somewhere to the left), with the previous spot empty. 
Note $0\leq L\leq k$ and those cars occupy spots $z-L,z-L+1,\ldots, z-1$, and spot $z-L-1$ is empty.
From the $k$ cars which parked to the left of spot $z$, we must select $L$, which we can do in $\binom{k}{L}$ ways. 
The $L$ cars can have $|\IPF_{L,L}(\ell)|=|\IPF_L(\ell)|$ possible preferences so as to park in the spots required.
Then the remaining $k-L$ cars must park in spots $1,2,\ldots, z-L-2$, as we must leave spot $z-L-1$ empty. 
Those cars can have $|\IPF_{k-L,z-L-2}(\ell)|$ preferences.
To ensure car $n$ parks in spot $z$, we note that the possible preferences depend on both $\ell$ and $L$. 
Namely, car $n$ can have $\min(\ell,L)+1$ preferences. 
Hence, the number of ways to park $k$ cars to the left of spot $z$ so as to satisfy all of these conditions, and so that car $n$ park in spot $z$,  is given by 
\[\binom{n-1}{k}\left(\sum_{L=0}^{k}\binom{k}{L}\cdot (\min(\ell,L)+1)\cdot |\IPF_{k-L,z-L-2}(\ell)|\cdot |\IPF_{L}(\ell)|\right).\]

Next we consider the cars that park to the right of spot $z$. Note that those cars are predetermined by the choice of the cars which park to the left of spot $z$.
Namely, we need not select them, they simply are the complement of the choice of the $k$ cars which parked to the left of spot $z$.
Those $n-1-k$ cars park in spots $z+1,z+2,\ldots,m$, hence the number of parking preferences is given by 
$|\IPF_{n-1-k,m-z}(\ell)|$.

Now we must sum over 
$0\leq k\leq z-1$ and over $1\leq z\leq m$ which yields the desired formula.
\end{proof}

\section{Nondecreasing \texorpdfstring{$\ell-$}-interval rational parking functions}\label{sec:nondecreasingLInt}

In this section, we give recursive formulas to enumerate the set 
$\NDIPF{m,n}{\ell}$ consisting of nondecreasing $\ell$-interval rational parking functions. 
This enumeration relies on having some empty spots after cars have parked, and thus requires the restriction that there are more spots than cars. 
The case in which the number of cars equals the number of spots is treated separate and in connection to enumerating Dyck paths.

As before, we simplify notation and let $\IPF_{n,n}^\uparrow(\ell)= \IPF_n^{\uparrow}(\ell)$ when we have the same number of cars as parking spots. 

Utilizing the outcome map for rational parking functions (Definition \ref{def:outcome}), we can count nondecreasing $\ell$-interval rational parking functions as follows.

\begin{theorem}\label{thm:nondecreasing ell rational}
  Let $m,n\in\PP$  with $m> n$ and $\ell\in\NN$. Then 
 \[ |\NDIPF{n,m}{\ell}|=\sum_{Z=\{z_1<z_2<\cdots<z_{m-n}\}\subset[m]}\quad \prod_{i=1}^{m-n}|\NDIPF{z_{i}-z_{i-1}-1}{\ell}|\]
 where $|\NDIPF{0}{\ell}|=1$.
\end{theorem}
\begin{proof}
Given that the preference lists are elements in $[m]^n$ and are nondecreasing, the outcome map consists of $m-n$ zeros and the values in $[n]$ appear in sequential order from left to right.
Thus each subset $Z$ of $[m]$, with size $m-n$, gives a choice of indices where the zeros in the outcome map appear. 
Then each car must prefer a spot in $[m]\setminus Z$ to achieve the correct outcome. 

More precisely, fix $Z=\{z_1<z_2<\cdots<z_{m-n}\}\subset[m]$. We now count the preferences of the $n$ cars by considering the preferences of cars that park in consecutive spots. 
The number of cars parking in consecutive spots is determined by the set $Z$. 
Namely, for each $i\in[m-n]$, the value $z_{i}-z_{i-1}-1$ (where $z_0=0$) accounts for the number of cars parked between zeros at position $z_{i-1}$ and $z_{i}$.

Since the tuple is nondecreasing, we know the cars all park in sequential order on the street. 
Thus, the preferences for the cars is the product of the number of nondecreasing $\ell$-interval parking functions on parking lots with sizes given by the values $z_{i}-z_{i-1}-1$ for each $i\in[m-n]$. 

Therefore, the total number of elements in $\NDIPF{m,n}{\ell}$ is given by 
\[|\NDIPF{n,m}{\ell}|=\sum_{Z=\{z_1<z_2<\cdots<z_{m-n}\}\subset[m]}\prod_{i=1}^{m-n}|\NDIPF{z_{i}-z_{i-1}-1}{\ell}|.\]
Note that whenever $z_i=z_{i-1}+1$, then $|\NDIPF{z_i-{z_i-1}-1}{\ell}|=|\NDIPF{0}{\ell}|=1$ accounting for the empty preference list.
\end{proof}

Note that Theorem \ref{thm:nondecreasing ell rational} only works whenever $m>n$, as it requires the existence of a set of size $m-n$. Hence, we now consider a classical bijection between parking functions and Dyck paths to give another formula for $|\NDIPF{n,m}{\ell}|$ in both cases when $m>n$ and $m=n$.

\subsection{A connection to lattice paths}
In light of the results presented in the previous section, we now connect our work on nondecreasing $\ell$-interval parking functions with lattice paths. 
\begin{remark}\label{dyck catalan} It is well-known that nondecreasing parking functions of length $n$ are in bijection with Dyck paths of length $2n$, which are counted by the $n$th Catalan number $C_n=\frac{1}{n+1}\binom{2n}{n}$, see OEIS \href{https://oeis.org/A000108}{A000108}. 
This implies that $|\PF_n^\uparrow|=C_n.$
\end{remark}
Inspired by Remark \ref{dyck catalan}, we establish 
a bijection between $\NDIPF{n}{\ell}$
and Dyck paths of length~$2n$ with a maximal height $\ell+1$. Before doing so, we provide the following definitions and notation.

\begin{definition}Recall that a Dyck path of length $2n$ is a lattice path from $(0,0)$ to $(n,n)$ consisting of $(1,0)$ and $(0,1)$ steps, respectively referred to as east and north steps, which lie above the diagonal $y=x$. 
We let $\Dyck{n}{}$ denote the set of Dyck paths of length $2n$. A Dyck path of length $2n$ is said to have height $k$ if all steps lie on or between the lines $y=x$ and $y=x+k$, and at least one step touches the line $y=x+k$. 
We let $\Dyck{n}{k}$ denote the set of all Dyck paths with length $2n$ and height at most $k$.
\end{definition}
\begin{remark}\label{rem:north-east}
In the $n\times n$ lattice, step $k$ of the Dyck path refers to entry $w_k$ in the Dyck word associated to $w$, which we define shortly. Note that the height of a Dyck path at step $k$ is the number of North steps minus the number of East steps in its first $k$ steps. 
Moreover, the height of a Dyck path is its maximum height over all of its steps.
\end{remark}

Following \cite{armstrong2016rational}, we recall that Dyck paths can be encoded via a \emph{Dyck word} of length $2n$ consisting of $n$ letters~$N$, representing north steps, and $n$ letters $E$, representing east steps, satisfying
\begin{enumerate}
\item the word begins with $N$, and 
\item at every step the number of north steps is greater than or equal to the number of east steps.
\end{enumerate}

In Figure \ref{fig:Dyck ex}, we provide two Dyck paths (as  lattice paths) and their corresponding Dyck words. 
Note that $D_1\in \Dyck{6}{2}$, while $D_2\in \Dyck{6}{3}$.

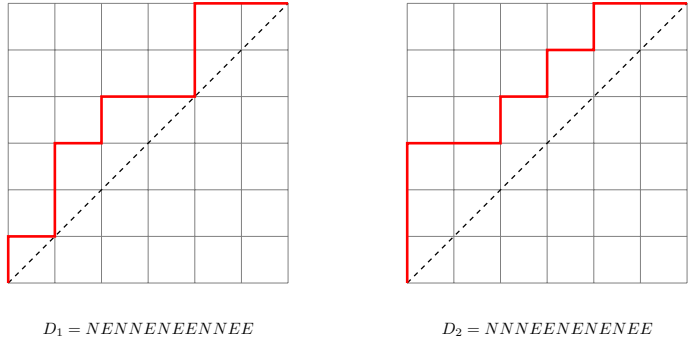
\begin{figure}[h]
\centering
\resizebox{1.5in}{!}{
\begin{tikzpicture}
\draw[step=1cm,gray, thin] (0,0) grid (6,6);
    \draw[thick,dashed](0,0)--(6,6);
    \draw[ultra thick,red](0,0)--(0,1)--(1,1)--(1,3)--(2,3)--(2,4)--(4,4)--(4,6)--(6,6);
    \node at (3,-1) {$D_1=NENNENEENNEE$};
\end{tikzpicture}}
\qquad\qquad
\resizebox{1.5in}{!}{
\begin{tikzpicture}
    \draw[step=1cm,gray, thin] (0,0) grid (6,6);
    \draw[thick,dashed](0,0)--(6,6);
    \draw[ultra thick,red](0,0)--(0,3)--(2,3)--(2,4)--(3,4)--(3,5)--(4,5)--(4,6)--(6,6);
    \node at (3,-1) {$D_2=NNNEENENENEE$};
\end{tikzpicture}}
\caption{Examples of two Dyck paths and their corresponding Dyck words.}\label{fig:Dyck ex}
\end{figure}

We now recall the classical bijection between Dyck paths and nondecreasing parking functions. 

\begin{definition}\label{classical bijection}
Let  $\varphi:\Dyck{n}{}\to\NDPF{n}$  be defined as follows: If $w=w_1w_2\cdots w_{2n}$ is the Dyck word corresponding to the Dyck path $D\in\Dyck{n}{}$, then $\varphi(w)=(a_1,a_2,\ldots,a_n)$, where  for each $i\in[n]$, we set
$a_i$ equal to one more than the number of east steps appearing to the left of the $i$th north step in $w$. Namely, for each $i\in[n]$, if $\mathrm{N}(i)$ is the position of the $i$th north step in $w$, then 
\[a_i=1+|\{j \,: \, j<\mathrm{N}(i) \mbox{ and } w_j=E \}|.\]
\end{definition}
We illustrate the bijection in Definition \ref{classical bijection} through the following example.

\begin{example}
We illustrate the map $\varphi:\Dyck{4}{}\to\NDPF{4}$ of Definition \ref{classical bijection} below: 
\begin{center}
\begin{tabular}{cccc}
$\varphi(NNNNEEEE)=(1,1,1,1)$,&
$\varphi(NNNENEEE)=(1,1,1,2)$,\\
$\varphi(NNNEENEE)=(1,1,1,3)$,&
$\varphi(NNENNEEE)=(1,1,2,2)$,\\
$\varphi(NENNNEEE)=(1,2,2,2)$,&
$\varphi(NNENENEE)=(1,1,2,3)$,\\
$\varphi(NNNEEENE)=(1,1,1,4)$,&
$\varphi(NENNENEE)=(1,2,2,3)$,\\
$\varphi(NNEENNEE)=(1,1,3,3)$,&
$\varphi(NNENEENE)=(1,1,2,4)$,\\
$\varphi(NNEENENE)=(1,1,3,4)$,&
$\varphi(NENNEENE)=(1,2,2,4)$,\\
$\varphi(NENENNEE)=(1,2,3,3)$,&
$\varphi(NENENENE)=(1,2,3,4)$.\\ 
\end{tabular}
\end{center}

\end{example}

\begin{theorem}\label{them:bijection nondecreasing to dyck paths}
The sets $\NDIPF{n}{\ell}$ and $\Dyck{n}{\ell+1}$ are in bijection.
\end{theorem}

\begin{proof}
We have that $\Dyck{n}{\ell+1}\subseteq\Dyck{n}{}$ and $\NDIPF{n}{\ell}\subseteq \NDPF{n}$ for any $\ell\geq 0$. 
Since $\varphi:\Dyck{n}{}\to \NDPF{n}$ is a bijection, we know that $\varphi$ restricted to $\Dyck{n}{\ell+1}$ bijects onto its image $\varphi(\Dyck{n}{\ell+1})$. 
Hence, it suffices to show that $\varphi(\Dyck{n}{\ell+1})=\NDIPF{n}{\ell}$.
 We prove this by showing both set containments:
\begin{enumerate}
    \item $\varphi(\Dyck{n}{\ell+1})\subseteq \NDIPF{n}{\ell}$,
    and 
    \item $\NDIPF{n}{\ell}\subseteq \varphi(\Dyck{n}{\ell+1})$.
\end{enumerate}

For (1). Let $D\in \Dyck{n}{\ell+1}$ with corresponding Dyck word $w$.
Assume, for sake of contradiction, that in $\varphi(w)$ there exists a car $j\in[n]$ (which is the first such car) parking further than $\ell$ spots away from its preference $a_j$.
Recall that since $\varphi(w)$ is nondecreasing and a parking function, we must have $a_j\leq j$, as $a_j>j$ would imply the existence of a gap in the street and not all cars would park.
Since $\varphi(w)$ is a nondecreasing parking function, we know car $j$ must park in spot $j$. 
However, by assumption
\begin{equation}j-a_j\geq \ell+1,\label{inequality}\end{equation} as car $j$ with preference $a_j$ parked further than $\ell$ spots from its preference.
By definition of $\varphi$, we know 
$a_j=1+|\{x \,: \, x<\mathrm{N}(j) \mbox{ and } w_x=E \}|$. In what follows, let   $|\{x \,: \, x<\mathrm{N}(j) \mbox{ and } w_x=E \}|=\#E$. Replacing $a_j$ with $\#E+1$ in Equation (\ref{inequality}) yields
\begin{equation}
j-\#E-1\geq \ell+1.\label{eq:yay}
\end{equation}
At this point in $w$, we have taken $j$ many $N$ steps and the height of our Dyck path at step $j+\#E$ is $\#N-\#E$. 
From  inequality (\ref{eq:yay}), this means that  $\#N-\#E-1\geq \ell+1$, which implies that $\#N-\#E\geq \ell+2$. 
This contradicts that the height of this Dyck path is at most $\ell +1$.
Therefore, $\varphi(w)\in\NDIPF{n}{\ell}$ as desired.

For (2). 
Let $\alpha=(a_1,a_2,\ldots,a_n)\in\NDIPF{n}{\ell}\subseteq \NDPF{n}$.
Then $a_i\leq a_{i+1}$ for all $i\in[n]$ and (as the cars park in order 1 to $n$) we have that $i-a_i\leq \ell$.
For each $i\in[n]$, let \[\mathcal{B}_i=|\{j\in[n]: a_j=i\}|.\]
For each $\mathcal{B}_i$ construct a subword consisting of $|\mathcal{B}_i|$ north steps followed by an east step and denote~it 
\[x_i=\underbrace{N\cdots N}_{|\mathcal{B}_i|}E,\] 
If $\mathcal{B}_i=\emptyset$, then  $x_i=E$.
We now construct $w$ by concatenating $x_1,x_2,\ldots, x_n$ so that
\begin{equation}
w=x_1x_2\cdots x_n.\label{eq:new w}
\end{equation}
We now claim that: 
\begin{enumerate}
    \item[(i)] $w$ is a Dyck word corresponding to a Dyck path in $\Dyck{n}{}$,
    \item[(ii)] $\varphi(w)=\alpha$, and
    \item[(iii)] $w$ is a Dyck word corresponding to a Dyck path in $\Dyck{n}{\ell+1}$.
\end{enumerate} 
For (i): 
We need to check that $w$ has length $2n$, and that for any $i\in[2n]$, there are more N (north steps) than E east steps in $w_1\cdots w_i$.
Note that $\sum_{i=1}^n|\mathcal{B}_i|=n$ means that $w$ has $n$ north steps $N$, and for every $i\in[n]$ we have inserted a single east step $E$. 
Thus $w$ has length $2n$. 
Moreover, given that $\alpha$ is an $\ell$-interval parking function, it is a parking function, and hence $w$ satisfies $\#N-\#E\geq 0$ at every step.
Thus $w\in \Dyck{n}{}$.

For (ii). 
Given $w$ as in Equation (\ref{eq:new w}), define 
\[a_i'=1+|\{j\in[2n]: j<N(i)\mbox{ and } w_j=E\}|.\]
We claim $a_i=a_i'$ for all $i\in[n]$.
Let $b_i=|\mathcal{B}_i|$ for all $i\in[n]$.
The number of $1$ entries in $\alpha$ is $b_1$. 
Hence, for $1\leq j\leq b_1$, the number of East steps preceding the $j$th North step is exactly 0. 
Thus $a_j'=1+0=1$ whenever $1\leq j\leq b_1$. 
Therefore $a_j=a_j'$ for all $1\leq j\leq b_1$.
Now consider $m=a_{b_1+1}>a_{b_1}=1$. 
In $w$ we have inserted $m-1$ east steps $E$ into $w$ before the first $N$ is inserted from $\mathcal{B}_{m}$. However, \[a_{b_1+1}'=1+|\{j\in[2n]: j<N(b_1+1)\mbox{ and } w_j=E\}| = 1+ (m-1) = m = a_{b_1+1},\] 
as desired. 
Iterating this process for the remaining blocks $\mathcal{B}_{m+1},\ldots,\mathcal{B}_n$, yields that $a_j=a_j'$ for all remaining indices $b_1+1< j\leq n$. 

For (iii). It now suffices to note that from part (1) in this proof, we have shown by contradiction that if $w$ does not correspond to a Dyck path with maximum height $\ell+1$, then $\varphi(w)$ cannot be in $\NDIPF{n}{\ell}$. 
\end{proof}

The enumeration of Dyck paths based on their exact height is well-known (OEIS \href{https://oeis.org/A080936}{A080936})
and is credited to~\cite{Grossman}, although we were unable to access the article to confirm this.
Thanks to a comment from Emeric Deutsch posted on 
OEIS \href{https://oeis.org/A080936}{A080936} we found the following:
By \cite[Page 37-38]{kreweras1970eventails} the 
number of Dyck paths with semilength $2n$ and height at most~$k$, denoted $H(k)$, is given by
\begin{equation}\label{eq:gen}
    H(k)=[t^k]\left(\frac{f_k(t)}{f_{k+1}(t)}\right),
\end{equation}
which is the coefficient of $t^k$ in the generating function $f_k(t)/f_{k+1}(t)$, 
where $f_0(t)=f_1(t)=1$  and 
$f_{k+1}(t) = f_k(t) - t\cdot f_{k-1}(t)$. Note that the recurrence relation for the polynomials $f_k(t)$ was given in \cite[Page 36, Equation 14]{kreweras1970eventails}.


Given the bijection between nondecreasing $\ell$-interval  parking functions of length $n$ and Dyck paths of semilength $2n$ with height at most $\ell+1$, and Equation (\ref{eq:gen}), we arrive at the following result.

\begin{corollary}\label{cor:nondec_form}
For $k\geq 0$, the number of nondecreasing $\ell$-interval parking functions of length $n$ is  $H(\ell+1)$, as defined in Equation (\ref{eq:gen}).
\end{corollary}

We can extend this result naturally to  $\ell$-interval rational parking functions, which are in bijection with rational $(n,m)$-Dyck paths (see \cite{ceballos2018signature} for an overview of rational Dyck paths and other Catalan objects, and see \cite{christensenGeneralizationParkingFunctions2020} for a relationship between $k$-Naples parking functions and $k$-lattice paths). Let us continue with some definitions and notation.

\begin{definition}Recall that a $(n,m)$-Dyck path of length $n+m$ is a lattice path from the origin $(0,0)$ to the lattice point $(m,n)$ consisting of $(1,0)$ and $(0,1)$ steps, respectively referred to as east and north steps, which lie above the diagonal $y=\frac{n}{m}x$. 
We let $\Dyck{n,m}{}$ denote the set of Dyck paths of length $n+m$. 
\end{definition}

\begin{definition}
Given an $(n,m)$-Dyck path with corresponding word
$w=w_1w_2\ldots w_{n+m}$
the height of $w$ at step 
$1\leq i\leq n+m$, is define by 
\[h(i)=\#\{j: j\leq  i\mbox{ and $w_j=N$}\}-\#\{j: j\leq  i\mbox{ and $w_j=E$}\}.\]
Then the height of $w$ is \[h(w)=\displaystyle\max_{1\leq i\leq n+m}h(i).\]

We let $\Dyck{n,m}{k}$ denote the set of all $(n,m)$-Dyck paths with length $n+m$ and height at most~$k$.
\end{definition}

\begin{example}\label{exrat}
    Figure \ref{fig:exrat} gives an example of a rational $(5,8)$-Dyck path whose height is 2.
    \begin{figure}[ht]
    \centering

\resizebox{2.5in}{!}{
\begin{tikzpicture}
    \draw[step=1cm,gray, thin] (0,0) grid (8,5);
    \draw[thick,dashed](0,0)--(8,5);
    \draw[ultra thick,red](0,0)--(0,1)--(1,1)--(1,3)--(4,3)--(4,4)--(5,4)--(5,5)--(8,5);
\end{tikzpicture}
}
    
    \caption{A $(5,8)$-Dyck path with word NENNEEENENEEE with height 2.}
    \label{fig:exrat}
    \end{figure}
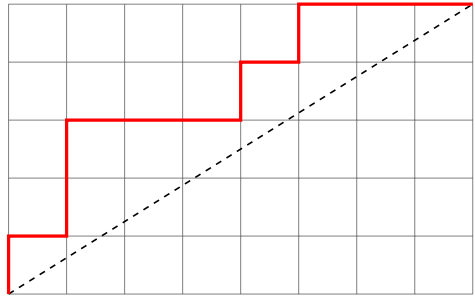

\end{example}
\begin{definition}[Pages 761-762 in \cite{ceballos2018signature}]\label{rational bijection}
Let  $\varphi:\Dyck{n,m}{}\to\NDPF{n,m}$  be defined as follows: If $w=w_1w_2\cdots w_{n+m}$ is the Dyck word corresponding to the Dyck path $D\in\Dyck{n,m}{}$, then $\varphi(w)=(a_1,a_2,\ldots,a_n)$, where  for each $i\in[n]$, we set
$a_i$ equal to one more than the number of east steps appearing to the left of the $i$th north step in $w$. Namely, for each $i\in[n]$, if $\mathrm{N}(i)$ is the position of the $i$th north step in $w$, then 
\[a_i=1+|\{j \,: \, j<\mathrm{N}(i) \mbox{ and } w_j=E \}|.\]
\end{definition}

\begin{example}[Example \ref{exrat} continued]
Observe \[\varphi(NENNEEENENEEE)=(1,2,2,5,6)\in \NDPF{5,8}.\]   
\end{example}

As expected, the function $\varphi:\Dyck{n,m}{}\to\NDPF{n,m}$ is a bijection, and we omit the proof as it is analogous to that of the case where $n=m$. We now focus on the following result.

\begin{theorem}\label{rational pf to dyck paths}
     The sets $\IPF^{\uparrow}_{n,m}(\ell)$ and $\Dyck{n,m}{\ell+1}$ are in bijection.
\end{theorem}

\begin{proof}
    If $n=m$, the result is proven in Theorem~\ref{them:bijection nondecreasing to dyck paths}.
    If $n<m$, we will use a similar argument as in Theorem~\ref{them:bijection nondecreasing to dyck paths} to show that $\varphi$ from Definition~\ref{rational bijection} applied to any $D \in \Dyck{n,m}{\ell+1}$ will yield a preference list of $n$ cars parking on $m$ spots, each car $j$ displaced by at most $\ell$ spots from its preference $a_j$.
    First, we show that $\varphi(\Dyck{n,m}{\ell+1})\subseteq \IPF^{\uparrow}_{n,m}(\ell)$. Let $D \in \Dyck{n,m}{\ell+1}$ with corresponding word $w$.
    Assume there exists a car $j \in \varphi(w)$, such that $j-a_j \geq \ell+1$.
    As in the proof of Theorem~\ref{them:bijection nondecreasing to dyck paths}, this would imply that the height of $D$ is greater than $\ell+1$, contradicting $D \in \Dyck{n,m}{\ell+1}$.

To complete the proof it suffices to show that 
    $\NDIPF{n}{\ell}\subseteq \varphi(\Dyck{n}{\ell+1})$.
    The proof of this follows by an analogous argument as that presented in part (2) within the proof of Theorem~\ref{them:bijection nondecreasing to dyck paths}. The only difference is that in the construction of the word $w$ in Equation (\ref{eq:new w}) there are $n$ north steps and $m$ east steps and the Dyck path has height at most $\ell+1$.
\end{proof}

We were unable to find an enumeration of $(n,m)$-Dyck paths based on their height. It remains an open problem to give a closed formula for the set of nondecreasing $\ell$-interval rational parking functions.

\subsection{The members of \texorpdfstring{$\NDIPF{n}{\ell}$}~\,that are also Fubini rankings}\label{sec:NonDecAndFubini}
We now consider the intersection of nondecreasing $\ell$-interval parking functions which are also Fubini rankings (Definition \ref{def:fubini_rank}). The motivation for considering these sets is the recent work of Elder, Harris, Kretschmann, and Mart\'inez Mori, who 
enumerated \textit{unit Fubini rankings} which are unit interval parking functions of length $n$ that are also Fubini rankings.
They established the following enumerative result \cite[Theorem 3.3]{boolean}.
\begin{theorem}
    The number of unit Fubini rankings on $n$ competitors is given by the Fibonacci number $F_{n+1}$, where $F_{n+1}=F_n+F_{n-1}$ for $n\geq 3$ with initial values $F_1=F_2=1$.
\end{theorem}
Following that naming convention, we henceforth, refer to $\ell$-interval parking functions of length $n$ which are also Fubini rankings as \emph{$\ell$-interval Fubini rankings}. We let $\FR{n} (\ell)$ denote the set of $\ell$-interval Fubini rankings with $n$ competitors and we let $\FR{n}^\uparrow(\ell)$ denote the subset of $\FR{n}(\ell)$ with nondecreasing entries. 
We now provide an analogous enumerative result for the $\ell$-interval Fubini rankings, which are enumerated by generalizations of the Fibbonacci numbers.

\begin{theorem}\label{thm:FibonacciOrdern}
If $n\in \PP$ and $\ell\in\NN$ with $\ell+1\leq n$, then  
\[|\FR{n}^\uparrow(\ell)|=\sum^{\ell}_{x=0} \left| \FR{n-x-1}^{\uparrow}(\ell)\right|,\]
where $|\FR{0}(\ell)|=1$.
\end{theorem}

\begin{proof}
To begin note that if $n=0$, then $\FR{0}(\ell)=\{()\}$ where $()$ denotes the empty list. Thus $|\FR{0}(\ell)|=1$ as claimed. 
To count the elements of $\FR{n}^\uparrow(\ell)$ we consider the possible ending value $a_n$ which is the value at the $n$th coordinate. 
If $1\leq a_n\leq n-\ell-1$, then car $n$ would park in spot $n$ (as any nondecreasing parking function parks the cars in order). 
This would mean car $n$ parked at least $n-(n-\ell-1)=\ell+1$ spots away from its preference, which would mean the tuple is not a $\ell$-interval parking function.
However, if $n-\ell\leq a_n\leq n$, then car $n$ would park at most $\ell$ spots away from its preference.

Now we can count the elements $\alpha\in\FR{n}^\uparrow(\ell)$ by considering the value $a_n$.
Either $a_n$ is a tie or not. 
If it is not a tie then $a_n=n$, and there are $|\FR{n-1}^\uparrow(\ell)|$ 
many possibilities for the values in the first $n-1$ entries of $\alpha$.
If $a_n$ is a tie, then it can tie with at most $\ell$ other players, so that there are $\ell+1$ entries in $\alpha$ with that tied rank. Note that this ensures that as a parking function, $a_n$ allows car $n$ to park in spot $n$ which would be at most $\ell$ spots away from its preference.
Now if $x$ is the number of competitors that tie with $a_n$ (for a total of $x+1$ tied competitors), then they all earn rank $a_n=n-x+1$.
Then, there are $|\FR{n-x-1}^\uparrow(\ell)|$ 
many possibilities for the values in the first $n-x-1$ entries of $\alpha$.
Thus, the number of nondecreasing $\ell$-interval Fubini rankings is given by 
\[|\FR{n}^\uparrow(\ell)|=\sum_{x=0}^{\ell}|\FR{n-1-x}^\uparrow(\ell)|.\]
\end{proof}

In Table \ref{tab:seqs2}, we give data for the cardinalities of the set $\FR{n}^\uparrow(\ell)$ when  $0\leq n\leq 11$ and $0\leq \ell\leq 6$. We direct interested readers to \cite{milesGeneralizedFibonacciNumbers1960}). 
\begin{table}[h]
\centering
\begin{tabular}{|c||c|c|c|c|c|c|c|c|c|c|c|c|} \hline

$\ell\setminus n$ &1&2&3&4&5&6&7&8&9&10&11&OEIS  \\ \hline\hline
1  & 1& 1& 2& 3& 5& 8& 13& 21& 34& 55& 89 &  \href{https://oeis.org/A000045}{A000045}\\ \hline 
2 & 1& 1& 2& 4&  7& 13& 24& 44& 81& 149& 274 &  \href{https://oeis.org/A000073}{A000073}\\ \hline 
3 & 1& 1& 2& 4& 8&  15& 29& 56& 108& 208& 401 & \href{https://oeis.org/A000078}{A000078} \\ \hline 
4 & 1& 1& 2& 4& 8& 16& 31& 61& 120& 236& 464 & \href{https://oeis.org/A001591}{A001591}\\ \hline 
5 & 1& 1& 2& 4& 8& 16& 32& 63& 125& 248& 492 & \href{https://oeis.org/A000383}{A000383}\\ \hline 
6 & 1& 1& 2& 4& 8& 16& 32& 64& 127& 253& 504 & \href{https://oeis.org/A122189}{A122189}\\ \hline 
\end{tabular}
\caption{Cardinalities of the set $ \FR{n}^\uparrow(\ell) $ for $0\leq n\leq 11$ and $0\leq \ell\leq 6$.}\label{tab:seqs2}
\end{table}

\begin{remark}
As noted at the start of this section, Elder, Harris, Kretschmann, and Mart\'inez Mori studied the set $\FR{n}(1)$ \cite{boolean}. They proved that $|\FR{n}(1)|$ was equal to the number of Boolean intervals in the weak order lattice of $\mathfrak{S}_n$ \cite[Theorem 1.2]{boolean}. It remains an open problem to determine if there is some correspondence between the set $\FR{n}(\ell)$ and other structures in the weak order lattice of $\mathfrak{S}_n$.
\end{remark}

\section{Unit interval rational parking functions and preferential arrangements}\label{sec:BPA}
The main goal in this section is to establish a bijection between the set of unit interval rational parking functions with $n$ cars and $m\geq n$ parking spots and the set of preferential arrangements on $[n]$ with $m-n$ bars.

\subsection{Preferential arrangements}

Let $n \in \NN$. 
A preferential arrangement on $[n]$ is 
a tuple in $[n]^n$ giving 
an ordered ranking of its elements. 
Numbers tied in rank appear adjacent to each other and form a \emph{block}. 
Parenthesis denote separate blocks in ascending order of rank from left to right.
The rank of a block (and that of its elements) is one plus the number of elements in blocks to the left of it.
The order of elements within a block does not matter. For notational convenience, we write them in ascending order. 

\begin{example}\label{ex:pa on 3}
The preferential arrangements of $[3]$ are
\begin{center}
\begin{tabular}{cccccccc}
    $(1)(2)(3)$,&
    $(2)(1)(3)$,&
    $(2)(3)(1)$,&
    $(3)(2)(1)$,&
    $(3)(1)(2)$,&
    $(1)(3)(2)$,&
    $(1~2)(3)$,\\
    $(1~3)(2)$,&
    $(3)(1~2)$,&
    $(2~3)(1)$,&
    $(1)(2~3)$,&
    $(2)(1~3)$,&and&
    $(1~2~3)$.&\\
     \end{tabular}
     \end{center}
\end{example}

\cite{Gross_preferential_arrangements} defines preferential arrangements as the number of ways $n$ distinguishable objects can be distributed into $r$ distinguishable boxes where $r$ varies $1\leq r\leq n$.
Now recall that Fubini rankings are rankings of $n$ competitors in which ties are allowed (Definition \ref{def:fubini_rank}). 
There is a bijection between preferential arrangements and Fubini rankings, which we prove next in Theorem~\ref{lem:funini and pa}.

\begin{theorem}\label{lem:funini and pa}
  Let $n,r\in \mathbb{Z}_+$ and let $\FR{n,r}$ 
  denote the set of Fubini rankings with $n$ competitors and $r$-distinct ranks, similarly let $\PA{n,r}$ denote the set of preferential arrangements of $n$ objects into $r$ boxes, labeled and placed in order $B_1 B_2 \cdots B_r$. If $\alpha=(a_1,a_2,\ldots,a_n)\in \FR{n,r}$, then let $I(k)=\{i\in [n]: a_i=k\}$. 
  The function 
$\varphi: \FR{n,r}\to \PA{n,r}$ where 
$\varphi(\alpha)=B_1B_2B_3\cdots B_r$ where the content of $B_j$ is given by $I(j)$ is a bijection.
\end{theorem}
\begin{proof}
Note the function is well-defined as $\alpha\in \FR{n,r}$ then $\alpha$ consists of $n$ indices whose values are between 1 and $n$, and there are exactly $r$ distinct numbers present in $\alpha$. This implies  $\varphi(\alpha)$  places all indices with the rank $k$ into box $B_k$. The result arranges the $n$ distinguishable objects (namely the numbers in $[n]$) into the $r$ boxes.\\

\noindent\textbf{For injectivity:} Let $\varphi(\alpha)=\varphi(\beta)$ and assume for contradiction that $\alpha=(a_1,a_2,\ldots,a_n)\neq \beta=(b_1,b_2,\ldots,b_n)$. Let $i\in[n]$ be the smallest index at which $a_i\neq b_i$.
Then $i$ is in $B_{a_i}$ in $\varphi(\alpha)$, while $i$ is in $B_{b_i}$ in $\varphi(\beta)$. A contradiction, as we assumed that $\varphi(\alpha)=\varphi(\beta)$.\\

\noindent\textbf{For surjectivity:} 
Let $B_1B_2\cdots B_r\in\PA{n,r}$. We construct a tuple in $[n]^n$ as follows:
If $i\in B_1$, then set $a_i=1$. 
If $i\in B_j$ with $1<j\leq r$, then set 
$a_i=1+\sum_{k=1}^{j-1} |B_k|$.
Let $\alpha=(a_1,a_2,\ldots,a_n)$.
Claim: $\alpha\in\FR{n,r}$. First note that $\alpha$ has length $n$ as $\sum_{j=1}^r|B_j|=n$. By construction, the tuple $\alpha$ has precisely $r$ distinct values i.e.,~$r$ ranks. 
We need to show that the ranks in $\alpha$ satisfy the needed condition for a Fubini ranking.
Note that there are $|B_1|$ entries in $\alpha$ with value one. Then the next rank appearing in $\alpha$ is precisely $1+|B_1|$, accounting for the fact that $|B_1|$ competitors tied for rank one. 
Note that then rank $1+|B_2|$ appears in $\alpha$ at index $i$ whenever $i\in B_2$. 
Continuing in this fashion, we then guarantee that all $r$ ranks are utilized and that no ranks are skipped. 
Thus $\alpha$ is a Fubini ranking, as claimed. 
It suffices to show that  $\varphi(\alpha)=B_1B_2\cdots B_r$. By definition, all the elements of $\alpha$ in $I(1)$ will be placed in $B_1$. Similarly, if $x$ is the next available rank in $\alpha$, all the elements with rank $x$ will be placed in $B_2$, repeating this process for each of the $r$ available ranks ensures that $\varphi(\alpha)=B_1B_2\cdots B_r$, as claimed.
\end{proof}

\begin{corollary}
The set of preferential arrangements on $[n]$ is in bijection with the set of  Fubini rankings with $n$ competitors.
\end{corollary}
\begin{proof}
    Given that $\PA{n}=\cup_{r=1}^{n}\PA{n,r}$ and $\FR{n}=\cup_{r=1}^{n}\FR{n,r}$, and the fact that the function $\varphi$ in Theorem, \ref{lem:funini and pa} preserves the number of distinct ranks/boxes, we have that $\varphi:\FR{n}\to\PA{n}$ defined analogously as in Theorem \ref{lem:funini and pa} is also a bijection.
\end{proof}

\begin{example}\label{ex:preferential}
We continue Example \ref{ex:pa on 3} and illustrate the bijection from Theorem~\ref{lem:funini and pa} below.
\begin{center}
\begin{tabular}{cccc}
$\varphi((1,2,3))=(1)(2)(3)$,&
$\varphi((2,2,1))=(3)(1~2)$,\\
$\varphi((2,1,3))=(2)(1)(3)$,&
$\varphi((3,1,1))=(2~3)(1)$,\\
$\varphi((3,1,2))=(2)(3)(1)$,&
$\varphi((1,2,2))=(1)(2~3)$,\\
$\varphi((3,2,1))=(3)(2)(1)$,&
$\varphi((1,3,1))=(1~3)(2)$,\\
$\varphi((2,3,1))=(3)(1)(2)$,&
$\varphi((2,1,2))=(2)(1~3)$,\\
$\varphi((1,3,2))=(1)(3)(2)$,&
$\varphi((1,1,1))=(1~2~3)$,\\
$\varphi((1,1,3))=(1~2)(3)$.&\\ 
\end{tabular}
\end{center}

\end{example}

Theorem \ref{lem:funini and pa} and Theorem \ref{thm:Hadaway} imply the following result. 
\begin{corollary}
If $n\geq 1$, then $|\PA{n}|=\Fub_{n}$, the $n$th Fubini number (OEIS \href{https://oeis.org/A000670}{A000670}).
\end{corollary}

\subsection{Unit interval rational parking functions and barred preferential arrangements}\label{sec:bijection things}

From \cite{unit_pf}, we know that the set of Fubini rankings are in bijection with the unit interval parking functions. In Theorem \ref{lem:funini and pa}, we establish that Fubini rankings are also in bijection with preferential arrangements. Thus, preferential arrangements are in bijection with unit interval parking functions. In fact more is true. In this section,  we establish that unit rational parking functions are in bijection with barred preferential arrangements.

Barred preferential arrangements were first introduced by \cite{ahlbachBarredPreferentialArrangements2013} and have been also studied by~\cite{nkonkobe2019generalised} and they are a generalization of preferential arrangements, which we define next.

\begin{definition}[Barred Preferential Arrangements]\label{def:BPA}
Consider a preferential arrangement $q \in\PA{n}$ and let $b \in \mathbb{N}$.
Separating the blocks of $q$ with a total of $b$ bars results in a barred preferential arrangement.
\end{definition}

Note that every way of inserting $b$ bars into $q$ (or into any preferential arrangement) yields a barred preferential arrangement. 
We denote the set of all barred preferential arrangements on $[n]$ with $b$ bars by $\BPA(n, b)$.

\begin{example}\label{ex:BPA(2,2)}
    Let $n=2$. The elements of $\PA{2}$ are $(1~2)$, $(1)(2)$, and $(2)(1)$.
    Let $b=2$. Then the elements of $\BPA(2,2)$ are
\begin{itemize}
    \item The preferential arrangement $(1~2)$ results in three barred preferential arrangements: \[||(1~2),\qquad |(1~2)|,\qquad (1~2)||.\]
    \item The preferential arrangement $(1)(2)$ results in five barred preferential arrangements: 
    \[||(1)(2),\qquad |(1)|(2),\qquad (1)||(2),\qquad (1)|(2)|,\qquad (1)(2)||, \qquad |(1)(2)|.\]
    \item The preferential arrangement $(2)(1)$ results in five barred preferential arrangements:\[||(2)(1),\qquad |(2)|(1),\qquad (2)||(1),\qquad (2)|(1)|,\qquad (2)(1)||, \qquad |(2)(1)|.\]
\end{itemize}
Thus, there are $|\BPA(2,2)|=3+6+6=15$ barred preferential arrangements of $[2]$ with two bars.
\end{example}
We now extend the definition of unit interval parking functions to the case where there are possibly more spots than cars.
\begin{definition}
    Let $n,m\in \PP$. We define the set of \emph{unit interval rational parking functions} with $n$ cars on $m\geq n$ spots where all the cars park at most one spot away from their preference. We denote this as $\IPF_{n,m}(1)$. 
\end{definition}

We now illustrate a correspondence, that we prove shortly, between barred preferential arrangements and unit interval rational parking functions. 
\begin{example}\label{ex:UIRPF_BPA}
We can associate a unit interval rational parking function to the  barred preferential arrangement
\[
        q = (2) \mid (3 \thinspace 5) \thinspace (1) \mid \mid (4)\in \BPA(5,3)
  \]
by letting both the parentheses and the bars denote the blocks. 
Parentheses denote nonempty blocks, while bars denote empty blocks. 
We then number these blocks from left to right. In this case there would be a total of 7 blocks: $B_1=(2)$, $B_2=|$, $B_3=(3~5)$, $B_4=(1)$, $B_5=|$, $B_6=|$, $B_7=(4)$.
We let $|B_i|$ be the number of elements in the block, when $B_i=|$, we define $|B_i|=1$.
Then for any element $i\in[5]$ in a nonempty block $B_j$ with $1\leq j\leq 7$, we let $a_i=1+\sum_{k=1}^{j-1}|B_k|$, and we let $\alpha=(a_1,a_2,a_3,a_4,a_5)$. Note this is an analogous construction to the map defined in Theorem~\ref{lem:funini and pa}.  
    Thus,
    the unit interval rational parking function corresponding to the barred preferential arrangement  $q$ is $\alpha=(5,1,3,8,3)\in\IPF_{5,8}(1)$. 
\end{example}

We can now state and prove the main result of this section.

\begin{definition}[Definition 2.7, \cite{unit_pf}]\label{blocks}
Given a parking function $\alpha = (a_1, \ldots, a_n) \in \PF_{n}$, let $\alpha' = (a'_1, \ldots, a'_n)$ be the weakly increasing rearrangement of $\alpha$.
The partition of $\alpha'$ as a concatenation $\alpha'=\pi_1|\pi_2|\cdots|\pi_m$, where $\pi_j$ begins at (and includes) the $j$th entry $\alpha'_i$ satisfying $\alpha'_i=i$, is called the \emph{block structure} of~$\alpha$. Each $\pi_i$ is called a \textit{block} of $\alpha$.
\end{definition}
Under the correspondence between parking functions and Dyck paths, we remark that the start of each block in the block structure of $\alpha$, corresponds to a return to the diagonal in the Dyck path.

In the case of unit interval parking functions, the block structure has the following structure:

\begin{corollary}[Corollary 3.1, \cite{boolean}]
\label{cor: block internals}
Let $\alpha \in \upf{n}$ and $b_1\,|\,b_2\,|\,\cdots\,|\,b_m$ be its block structure.
For each $j \in [m]$, let $i_j$ be the minimal element of $b_j$.
Consider any $j \in [m - 1]$.
If $|b_j| = 1$, then $b_j = (i_j)$ and $i_{j+1} = i_j + 1$.
Otherwise, if $|b_j| = 2$, then $b_j = (i_j, i_j)$ and $i_{j+1} = i_j + 2$.
Otherwise, $|b_j| \geq 3$, $b_j = (i_j, i_j, \underbrace{i_j + 1, i_j + 2, \ldots, i_j + |b_j| - 2}_{|b_j| - 2 \mbox{ terms }})$ and $i_{j+1} = i_j + |b_j|$.
\end{corollary}

This definition generalizes to rational parking functions.

\begin{definition}\label{def:rational_blocks}
    Let $\alpha \in \IPF_{n,m}(1)$ let $\alpha'=(a_1', a_2', \ldots , a_n)'$ be its weakly increasing rearrangement. 
    We define the \emph{rational block structure} of $\alpha$ as follows: $\pi_1=a_1'a_2'\cdots a_k'$ provided that $a_2'=a_1'$, $a_j'=a_1+j-2$ for all $3\leq j\leq k$ and $a_{k+1}>a_1+k-2$. 
    Then $\pi_2$ begins with $a_{k+1}'$, and is defined analogously. Constructing all of the blocks in this way yields the block structure of $\alpha$ which we denote $\pi_1\pi_2\pi_3\cdots\pi_r$, for some integer $r\geq 1$.
\end{definition}
Observe that we can then utilize Definition \ref{def:outcome} of the outcome map to determine which spots remain empty once the cars park.
\begin{example}\label{ex:map_part1}
Consider the unit rational parking function $\alpha=(5,1,3,8,3,8,9)\in\IPF_{7,10}(1)$, which has rational block structure \[\pi_1=1,\; \pi_2=33,\;\pi_3=5,\;\pi_4=889.\]
Now note that the outcome 
    \[
    \Out(\alpha) = (2,0,3,5,1,0,0,4,6,7)
    \]
implies that spots $2,6,7$ remain empty.   
\end{example}

We now define a map that uses the outcome and the rational block structure of a unit rational parking function to construct a barred preferential arrangement.

\begin{definition}\label{def:uirpf_bijection}
Let $n,m\in \PP$ with $m\geq n$ and let $\alpha\in\IPF_{n,m}(1)$ with rational block structure $\pi_1\pi_2\cdots\pi_r$. Let $|\pi_i|$ denote the number of values in $\pi_i$ for all $i\in[r]$.
Let $\Out(\alpha)=(b_1,b_2,\ldots,b_m)$ be the outcome of $\alpha$ (Definition \ref{def:outcome}), noting that there will be $m-n$ instances of zero in $\Out(\alpha)$.
We 
define the map \[\vartheta:\IPF_{n,m}(1)\to \BPA(n,m-n)\]
for $\alpha\in\IPF_{n,m}(1)$ with outcome $\Out(\alpha)$, construct a barred preferential arrangement through the following steps:
\begin{enumerate}
    \item Omit the parentheses and commas from the outcome $\Out(\alpha)$.
    \item Replace all zeros with bars, which yields a word with letters $1,2,\ldots,n$ appearing exactly once and with $m-n$ bars interspersed.
    \item Begin at the left of this  word and place $|\pi_1|$ values in a set of parenthesis. Then place the next $|\pi_2|$ values in a set of parenthesis. 
    Iterate this process, for all $i\in[r]$, until all values have been placed within a set of parenthesis.
\end{enumerate}
Denote the result of this process by $\vartheta(\alpha)$.
    \end{definition}
    
\begin{example}[Example \ref{ex:map_part1} Continued]
Recall that $\alpha =(5,1,3,8,3,8,9)$ satisfies \[\Out(\alpha) = (2,0,3,5,1,0,0,4,6,7),\] and has rational block structure \[\pi_1=1,\; \pi_2=33,\;\pi_3=5,\;\pi_4=889.\]
Then by Definition \ref{def:uirpf_bijection} we arrive at the following after each step:
\begin{enumerate}
    \item $2035100467$, 
    \item $2|351||467$, and 
    \item $(2)|(35)(1)||(467)$.
\end{enumerate}
Thus, 
\[
        \vartheta (\alpha) = (2) \mid (3 \thinspace 5) \thinspace (1) \mid \mid (4\thinspace 6 \thinspace 7)\in \BPA(7,3).
  \]
\end{example}

\begin{remark}\label{rem:welldefined}
    Note that the rational block structure of a unit interval rational parking function $\alpha$ gives information about the parking outcome. 
    Namely, cars with preferences lying within the same block of $\alpha$ consist of a sequence of parking preferences which ultimately park in consecutive order on the street, with the first car parking in their preference and all subsequent cars being displaced by precisely one unit. 
Note that the sum of sizes of the blocks is precisely $n$, and $n$ is also the number of nonzero entries in the outcome of $\alpha$. So in Step 3 of Definition \ref{def:uirpf_bijection}, we will be able to partition the $n$ numbers from left to right using the sizes of the blocks.
Now 
suppose that under the map $\vartheta$, a bar is placed lying within a set of parenthesis. This  would imply that there was a block of cars with preferences $aa(a+1)\cdots (a+j)$, yet in the outcome a spot among spots numbered $a$
 through $a+j+1$ would be empty. This yields a contradiction as cars with those preferences would have all parked in those spots. 
Therefore $\vartheta$ places all bars outside of all parenthesis, which ensures that the map and $\vartheta$ returns a barred preferential arrangement. 
\end{remark}

\begin{theorem}\label{thm:bijection_uirpf}
For any $n, m \in \mathbb{N}$. The map $\vartheta : \IPF_{n,m}(1) \to \BPA (n,m-n)$ from Definition \ref{def:uirpf_bijection} is a bijection.
\end{theorem}

\begin{proof}
\noindent \textbf{Well defined:} Note that the process defining the map $\vartheta$ produces a barred preferential arrangement since the outcome of any unit rational parking function will list every element in the set $[n]$ exactly once, all $m-n$ instances of the value $0$ become a bar, parenthesising using the sizes of the blocks in the rational block structure of $\alpha$, and noting (Remark \ref{rem:welldefined}) that no bars ever lie within a set of parenthesis, ensures that $\vartheta$ produces a preferential arrangement.
Thus, satisfying all requirements for  $\vartheta(\alpha)$ to be an element of $\BPA(n,m-n)$, as desired.

\noindent\textbf{Injective:}
Let $\alpha,\beta\in\IPF_{n,m}(1)$ and $\alpha\neq \beta$.
We have the following cases:
either $\alpha$ and $\beta $ have different outcomes or the same outcome.
\begin{enumerate}
    \item If $\alpha$ and $\beta$ have different outcomes, then neither barring or parenthesization would produce the same barred preferential arrangement. Therefore $\vartheta(\alpha)\neq \vartheta(\beta)$.
    \item Assume $\alpha$ and $\beta$ have the same outcome.
    In this way step 1 yields the same word. Step 2 yields the same barred word. And now we must consider the rational block structures of $\alpha$, denoted $\pi_1\pi_2\cdots\pi_r$ and $\beta$, denoted $\tau_1\tau_2\cdots\tau_s$. 
    If $r\neq s$, then this implies that the parenthesization of step 3 would yield distinct barred preferential arrangements. 
    Thus, $\vartheta(\alpha)\neq \vartheta(\beta)$.
    Further suppose that $r=s$. 
    If there exists a $i\in[r]$ such that $|\pi_i|\neq |\tau_i|$, then again the parenthesization of step 3 would yield distinct barred preferential arrangements. 
    Implying, that $\vartheta(\alpha)\neq \vartheta(\beta)$.
    Thus we must further suppose that $|\pi_i|=|\tau_i|$ for all $i\in[r]$.
    We now claim that in fact $\pi_i=\tau_i$ for all $i\in[r]$. 
    Suppose that this is not the case and let $j$ be the smallest index at which $\pi_j\neq \tau_j$.
    By construction of $\pi_j$ and $\tau_j$, we know that if $\pi_j\neq\tau_j$ then they must begin with a different value. 
    Without loss of generality,
    assume that $\min(\tau_j)>\min(\pi_j)$. 
    This implies that under $\alpha$ a car parks at spot $\min(\pi_j)$. 
    On the other hand, based on the block structure of $\alpha$
    and $\beta$ matching up to the $j$th block, and 
    $\min(\tau_j)>\min(\pi_j)$, implies that 
    under $\beta$ no car occupies spot $\min(\pi_j)$. 
    This contradicts the fact that $\Out(\alpha)=\Out(\beta)$.
    Together this shows that $\vartheta(\alpha)\neq\vartheta(\beta)$, as desired.
\end{enumerate}

\noindent \textbf{Surjective:}
Let $q\in\BPA(n,m-n)$.
Immediately from $q$ we can read off the outcome of a rational parking function. 
Namely, removing the parenthesis from $q$, replacing bars in $w$ with zeros, and placing commas between all values yields an outcome vector. Call this vector $x(q)=(x_1,x_2,\ldots,x_{m})$.
Moreover, let $Z(q)$ denote the set of indices at which $x(q)$ has zeros.

We construct the unit interval rational parking function $\alpha$, as follows:
For any $i\in [n]$, define $\ind(i)$ as the index in which the value $i$ appears in $x(q)$.
For all $i\in[r]$, let  $m_i=\ind(\min(B_i))$, that is the index at which the minimal value of a block appears in $x(q)$.

Fix $i$ and let $B_j$ be the unique block which contains $i$.
Then let 
\[a_i=\begin{cases}
m_j \mbox{if $i$ is the first entry in $B_j$}\\
m_j+y-2\mbox{if $i$ is the $y$th entry in $B_j$}.
\end{cases}\]
Let $\alpha=(a_1,a_2,\ldots,a_n)$.

To finish the surjective proof, we must establish the following claims:
\begin{enumerate}
    \item $\alpha\in\IPF_{n,m}(1)$,
    \item $\Out(\alpha)=x(q)$, and 
    \item $\vartheta(\alpha)=q$.
\end{enumerate}

\noindent For (1):
Note that $\alpha\in\IPF_{n,m}(1)$ because by construction its block structure satisfies the requirements of Definition \ref{def:rational_blocks}.   \\

\noindent For (2):
First we consider the outcome of $\alpha$.
If $a_i=m_j$, meaning $i$ is the first entry in $B_j$, then car $i$ parks in spot $m_j$. To show that $\Out(\alpha)=x(q)$ we establish two things: both $\Out(\alpha)$ and $x(q)$ have zeros at the exact same indices, and then we establish that they both have the value $i\in[n]$ at the same index.

\noindent \textbf{Claim 1:} No car parks on a spot in $Z(q)$ under $\alpha$.\\
This  claim holds because to begin with we only define $a_i$ on the set of indices $i\in[m]\setminus Z(q)$. Hence $\Out(\alpha)$ has zeros  precisely at $Z(q)$.

By Claim 1, we know that $\Out(\alpha)$ and $x(q)$ have zeros at the exact same set of indices, denoted by $Z(q)$.
Hence, we need only consider indices with nonzero values.
Moreover note the remaining nonzero entries in $\Out(\alpha)$ and $x(q)$ consist of the values in the set $[n]$.
Fix $i\in[n]$.
It suffices to show that the value $i$ appears in the same index in both $\Out(\alpha)$ and $x(q)$.
Suppose that $x_k=i$, that is $i$ appears as the $k$th entry in $x(q)$.
Note that $a_i$ is the preference of the $i$th car in $\alpha$. 
By construction, there are two options for the value of $a_i$ depending on whether it is the first entry in a block or a further along entry in a block. 
We consider each case separately.
\begin{itemize}
    \item Assume $i$ is the first value in $B_j$.
    Then $a_i=\ind(\min(B_j))$.
In this case car $i$ is the first car (and possibly the only car) with this preference. Implying that car $i$ parks in spot $\ind(\min(B_j))$.
This establishes that value $i$ appears in $\ind(\min(B_j))$
in $\Out(\alpha)$.

Now observe that since $i$ is the first entry in $B_j$ of $q$, then  $\min(B_j)=i$ and hence $\ind(\min(B_j))=\ind(i)=k$. 
Thus value $i$ appears as the $k$th entry in both $\Out(\alpha)$ and $x(q)$ as desired.

\item Assume $i$ is not the first value in $B_j$, but the $y$th value, with $y\geq 2$.
Then $a_i=m_j+y-2=\ind(\min(B_j))+y-2$.
Note that by construction of $\alpha$, car $i$ parks in spot $\ind(\min(B_j))+y-1$. 
Thus $i$ appears as entry $\ind(\min(B_j))+y-1$ in $\Out(\alpha)$.
Now lets consider the index at which $i$ appears in $x(q)$.
By assumption, $i$ appears as the $y$th entry in block $B_j$. 
Thus, $i$ appears as entry $\ind(\min(B_j))+y-1$ in $x(q)$.
\end{itemize}

\noindent For (3): using (1) and (2), all that we need to check is that the parenthesization of $\Out(\alpha)$ will produce $q$.

Let the rational block structure of $\alpha$ be $\pi_1\pi_2\cdots \pi_r$.
Reindex the nonbar boxes of $q$ by $B_1,B_2,\ldots B_s$.
In fact $r=s$, since by construction of $\alpha$, each $\pi_i$ was constructed to have size $B_i$ (under the reindexing).

Now we put all the steps above together. 
Since $\Out(\alpha)=x(q)$, steps 1 and 2 ensure we have a barred preferential arrangement in $\BPA(n,m-n)$. 
Then we add parenthesis satisfying $|\pi_i|=B_i$ for all $i\in[r]$ and so step 3 yields precisely the parenthesization of $q$.
Therefore, $\vartheta(\alpha)=q$.
\end{proof}

We conclude with some enumerative consequences of Theorem \ref{thm:bijection_uirpf} and the work of~\cite{ahlbachBarredPreferentialArrangements2013}.
\begin{theorem}{\cite[Theorem 2]{ahlbachBarredPreferentialArrangements2013}}  \label{theorem: counting formula}
If $m \geq 1$, then
\begin{equation}\label{x}
    |\IPF_{n,m}(1)| = \frac{1}{2^{m-n}(m-n)!} \sum_{i = 0}^{m-n} s(m-n+1,i+1) \thinspace |\IPF_{n + i}(1)|
\end{equation}
$s(n,k)$ is the unsigned Stirling number of the first kind (the number of permutations of $n$ elements 
having $k$ cycles).
\end{theorem}

\begin{theorem}{\cite[Theorem 3]{ahlbachBarredPreferentialArrangements2013}} \label{theorem: Pippenger formula}
If $m \geq 0$ and $n \geq 1$, then
\begin{equation}\label{xx}
   |\IPF_{n,m}(1)| = \sum_{k = 0}^{n} 
    k! \thinspace  S(n,k)
    \thinspace 
  \binom{m-n+k}{k}  
\end{equation}
where $S(n,k)$ are Stirling numbers of the second kind (OEIS \href{https://oeis.org/A008277}{A008277}).
\end{theorem}
Note that setting $m=n$ in Theorem \ref{xx} recovers Theorem \ref{thm:Hadaway}.
We conclude with the following illustrative example.

\begin{example}
In Example \ref{ex:BPA(2,2)}, we showed that $\BPA(2,2)=15$. Theorem \ref{thm:bijection_uirpf} implies that $|\IPF_{2,4}(1)|=15$.
Using Theorems \ref{theorem: counting formula} and Theorem \ref{thm:Hadaway}, we confirm
\begin{align*}
    |\IPF_{2,4}(1)|&=\frac{1}{2^2\cdot2!}\left(s(3,1)\cdot|\IPF_2(1)|+s(3,2)\cdot|\IPF_3(1)|+s(3,3)\cdot|\IPF_{4}(1)|\right)\\
    &= \frac{1}{8} (2\cdot3 + 3\cdot13 + 1\cdot75)\\
    &=15.
  \end{align*}
\end{example}

\acknowledgements
\label{sec:ack}
This material is based upon work supported by the National Science Foundation under Grant No. DMS-1659138 and the Sloan Grant under Grant No. G-2020-12592.

\nocite{*}
\bibliographystyle{abbrvnat}
\bibliography{sample-dmtcs}
\label{sec:biblio}

\end{document}